\documentclass[11pt]{article}

\usepackage{verbatim}
\usepackage{xfrac}
\usepackage{nicefrac}
\usepackage[colorlinks=true,linkcolor= red, citecolor=blue]{hyperref}
\usepackage{booktabs}
\usepackage[bf, small, labelsep=period]{caption}

\usepackage[english]{babel}
\usepackage{amsmath,amsthm,xy,array}
\usepackage{amssymb,float,caption,multicol,multirow,bm,xcolor}
\usepackage{bbm}

\usepackage{dsfont}

\advance\textwidth by 2cm

\advance\textheight by 2.5cm

\newtheorem{Assumption}{Assumption}

\theoremstyle{plain}
\newtheorem{theorem}{Theorem}[section]
\newtheorem{lemma}[theorem]{Lemma}
\newtheorem{corollary}{Corollary}
\newtheorem{proposition}[theorem]{Proposition}

\theoremstyle{definition}
\newtheorem{definition}[theorem]{Definition}
\newtheorem{example}[theorem]{Example}

\theoremstyle{remark}
\newtheorem{remark}[theorem]{Remark}

\numberwithin{equation}{section}

\def\frakJ{\mathfrak{J}}

\newcommand{\calF}{\mathcal{F}}
\newcommand{\calP}{\mathcal{P}}
\newcommand{\calA}{\mathcal{A}}

\newcommand{\Y}{\mathcal{Y}}

\newcommand{\abs}[1]{\left\vert#1\right\vert}
\newcommand{\set}[1]{\left\{#1\right\}}

\newcommand{\tv}{\tilde{\zeta}}

\newcommand{\tX}{\tilde{X}}

\newcommand{\tb}{\tilde{b}}

\newcommand{\tFil}{\tilde{\mathcal{F}}}
\newcommand{\tProb}{\tilde{\mathds{P}}}
\newcommand{\tExp}{\tilde{\mathds{E}}}
\newcommand{\Exp}{\mathds{E}}
\newcommand{\tOmega}{\tilde{\Omega}}

\newcommand{\Prob}{\mathbf{P}}

\newcommand{\U}{\mathcal{U}}

\newcommand{\R}{\mathds{R}}

\newcommand{\e}{\mathbb{E}}

\newcommand{\seq}[1]{\left<#1\right>}
\newcommand{\norm}[1]{\left\Vert#1\right\Vert}

\newcommand{\B}{\mathcal{B}}

\newcommand{\Fil}{\mathds{F}}

\makeatletter
\newcommand{\subjclass}[2][1991]{%
  \let\@oldtitle\@title%
  \gdef\@title{\@oldtitle\footnotetext{#1 \emph{Mathematics subject classification.} #2}}%
}
\newcommand{\keywords}[1]{%
  \let\@@oldtitle\@title%
  \gdef\@title{\@@oldtitle\footnotetext{\emph{Keywords.} #1.}}%
}
\makeatother

\title{\textsf{\textbf{Existence of optimal controls for stochastic Volterra equations}}}

\author{Andres  C\'ardenas \thanks{Universidad del Rosario, Bogot\'a, Colombia. \textbf{Email:} andres.cardenas@urosario.edu.co}\and
Sergio Pulido \thanks{Universit\'e Paris-Saclay, CNRS, ENSIIE,Univ Evry, Laboratoire de Math\'ematiques et Mod\'elisation d'Evry (LaMME), 23 Boulevard de France, 91037, Evry-Courcouronnes, France.\\ \textbf{E-mail:}  sergio.pulidonino@ensiie.fr}\and
 Rafael Serrano\thanks{Universidad del Rosario, Bogot\'a, Colombia. \textbf{Email:} rafael.serrano@urosario.edu.co}}

\begin{document}

\date{\today}

\maketitle



 
 \begin{abstract} 
We provide sufficient conditions that guarantee the existence of relaxed optimal controls in the weak formulation of stochastic control problems for stochastic Volterra equations (SVEs). Our study can be applied to rough processes that arise when the kernel appearing in the controlled SVE is singular at zero. The existence of relaxed optimal policies relies on the interaction between integrability hypotheses on the kernel and growth conditions on the running cost functional and the coefficients of the controlled SVEs. Under classical convexity assumptions, we can also deduce the existence of optimal strict controls.
\end{abstract}

\section{Introduction}

There has been a rapidly growing interest in studying stochastic Volterra equations (SVEs) of convolution type because they provide suitable models for applications that benefit from the memory and the varying levels of regularity of their dynamics. Such applications include, among others, turbulence modeling in physics \cite{barndorff2008time,barndorff2011ambit}, modeling of energy markets \cite{barndorff2013modelling}, and modeling of rough volatility in finance \cite{gatheral2018volatility}. 

In this paper, we consider finite-horizon control problems for SVEs of convolution type driven by a multidimensional Brownian motion with linear-growth coefficients and control policies with values on a metrizable topological space of Suslin type. We are particularly interested in singular kernels, such as fractional kernels proportional to $t^{H-\frac12}$ with $H \in (0,\frac{1}{2})$. These kernels are important because they allow modeling trajectories that are strictly less regular than those of classical Brownian motion. They have been used, for instance, in financial models with rough volatility which reproduce features of time series of estimated spot volatility \cite{gatheral2018volatility} and implied volatility surfaces \cite{alos2007short,bayer2016pricing}. 

Several studies have investigated the optimal control of SVEs. \cite{yong2006backward} uses the maximum principle method to obtain optimality conditions in terms of an adjoint backward stochastic Volterra equation. \cite{agram2015malliavin} also uses the maximum principle together with Malliavin calculus to obtain the adjoint equation as a standard backward SDE. Although the kernel considered in these papers is not restricted to convolution type, the required conditions do not allow the singularity of $K$ at zero. Recently, an extended Bellman equation has been derived in \cite{wang2021time} for the associated controlled Volterra equation.

The particular case of linear-quadratic control problems for SVEs, with controlled drift and additive fractional noise with Hurst parameter $H>1/2$,  has been studied in \cite{kleptsyna2003linear}. Similarly, in \cite{duncan2013linear} the authors consider a general Gaussian noise with an optimal control expressed as the sum of the well-known linear feedback control for the associated deterministic linear-quadratic control problem and the prediction of the response of a system to the future noise process. Recently, \cite{wang2018linear} investigated the linear-quadratic problem of stochastic Volterra equations by providing characterizations of optimal control in terms of a forward-backward system, but leaving aside its solvability, and under some assumptions on the coefficients that preclude (singular) fractional kernels of interest.

\cite{jaber2019linear} studied control problems for linear SVEs with quadratic cost function and kernels that are the Laplace transforms of certain signed matrix measures that are not necessarily finite. They establish a correspondence between the initial problem and an infinite dimensional Markovian problem on a certain Banach space. Using a refined martingale verification argument combined with the completion of squares technique, they prove that the value function is of linear quadratic
form in the new state variables, with a linear optimal feedback control, depending on nonstandard Banach space-valued Riccati equations. They also show that conventional finite dimensional Markovian linear-quadratic problems can approximate the value
function of the stochastic Volterra optimization problem.


We propose to study the existence problem using so-called relaxed controls in a weak probabilistic setting. This approach compactifies the original control system by embedding it into a framework in which control policies are probability measures on the control set, and the probability space is also part of the class of admissible controls. Thus, in this setting, the unknown is no longer only the control-state process but rather an array consisting of the stochastic basis and the control-state pair solution to the relaxed version of the controlled SVE.

In the stochastic case, relaxed control of finite-dimensional stochastic systems goes back to \cite{fleming-nisio}. Their approach was followed extensively by \cite{ElKaroui}, \cite{haussmann},  \cite{kurtz1998existence}, \cite{mezerdi2002necessary} and \cite{dufour2012existence}. Relaxed controls have also been used to study singular control problems  \cite{haussmann1995singular,kurtz2001stationary, andersson2009relaxed},  optimal control of stochastic PDEs \cite{brzezniak2013optimal}, mean-field games \cite{lacker2015mean,fu2017mean,cecchin2020probabilistic,benazzoli2020mean,bouveret2020mean,barrasso2020controlled},
mean-field control problems \cite{bahlali2017existence}, continuous-time reinforcement learning  \cite{wang2020reinforcement,wang2020continuous}, and optimal control of piece-wise deterministic Markov processes  \cite{costa2010average,costa2010policy,do2013continuous,bauerle2009mdp,bauerle2018optimal}


The main purpose of this paper is to provide a set of conditions that ensures the existence of optimal relaxed controls, see Theorem \ref{thm-main}. \textcolor{black}{We use methods that are similar to the approach employed by \cite{brzezniak2013optimal} for stochastic PDEs.} Our main contribution is that we allow kernels that are singular at zero, for instance, fractional Kernels proportional to $t^{H-\frac12}$ with $H \in (0,\frac {1}{2})$, and coefficients that are not necessarily bounded in the control variable. Under one additional assumption on the coefficients and cost function, familiar in relaxed control theory since the work of \cite{filippov}, we prove that the optimal relaxed value is attained by strict policies on the original control set, see Theorem \ref{thm-strict}. 

The paper is structured as follows. In Section \ref{sec:2} we establish some preliminary results on controlled stochastic Volterra equations (CSVEs). In Section \ref{sec:3} we describe the weak relaxed formulation of the control problem, state our main results, namely Theorems \ref{thm-main} and \ref{thm-strict}, and provide some examples. Section \ref{sec:4} contains the proofs of the main results. In Appendix \ref{A:A} we recall an important measurability result needed for the existence of optimal strict controls. Appendix \ref{A:B} contains an overview of the main results on relative compactness and limit theorems for Young measures that are used in the proofs of the main theorems.

\section{Controlled stochastic Volterra equations (CSVEs)}\label{sec:2}

Let $T>0$ and $d,d'\in\mathds N$ be fixed.  We consider the control problem of minimizing the cost functional of the form
\begin{equation}
\label{P0}
\e\left[ \int_0^{T}l(t,X_t,u_t )dt +G(X_T)  \right]
\end{equation}
subject to $X=(X_t)_{t\in [0,T]}$ being a $\R^d$-valued solution to the controlled  stochastic Volterra equation (CSVE) of the form
\begin{align}
	\label{EVC}
	X_t&=  x_0(t)+\int_0^tK(t-s)b(s,X_s,u_s)\,ds+ \int_0^tK(t-s)\sigma(s,X_s,u_s)\,dW_s, \ t\in [0,T]
\end{align}
over a certain class of control processes $(u_t)_{t\in [0,T]}$ taking values in a measurable control set $U.$ The  function $K\in L^2_{\textrm{loc}}(0,T;\R^{d\times d})$ is a given kernel, the initial condition $x_0$ is a deterministic $\R^d$-valued continuous function on $[0,T]$, and $(W_t)_{t\in [0,T]}$ is a $d'$-dimensional Brownian motion defined on a probability space $(\Omega,\mathcal F,\Prob)$ endowed with a filtration $\mathds F=(\calF_t)_{t\geq 0}$, satisfying the usual conditions. In our main existence results, we will consider solutions to \eqref{EVC} in a weak sense, see Definition \ref{def:relaxed_control}.

Throughout, we will assume the following condition on the kernel $K$:
\begin{Assumption}\label{Convcond} There exist $r \in (2,\infty)$ and $\gamma \in (0,2]$ such that $K\in L^{r}_{\rm loc}(\R_+;\R^{d\times d})$ and
\[
\int_{0}^{h}\abs{K(t)}^2 dt=\mathcal O(h^{\gamma}), \ \mbox{ and } \ \int_{0}^{T}\abs{K(t+h)-K(t)}^2 dt=\mathcal O(h^{\gamma}).
\]
\end{Assumption}
The following are examples of kernels that satisfy Assumption \ref{Convcond}:
\begin{enumerate}
	
	\item Let $K$ be locally Lipschitz. Then $K$ satisfies Assumption \ref{Convcond} with $\gamma=1$ and for any $r \in (2,\infty).$
	
	\item The fractional kernel $K(t)=t^{H-\frac12}$ with $H \in (0,\frac{1}{2})$ satisfies Assumption \ref{Convcond}  with  $r \in (2,\frac{2}{1-2H})$  and  $\gamma=2H.$
	
	\end{enumerate}
We consider, for now, a control set $U$ which is assumed to be a Hausdorff topological space endowed with the Borel $\sigma-$algebra $\B(U).$ We will assume later more specific conditions on $U.$
\begin{Assumption}\label{Assum1}
	$\phantom{assumption}$
	\begin{enumerate}
	
		\item The coefficients $b: [0,T]\times \R^d\times U\to \R^d$ and $\sigma: [0,T]\times \R^d\times U\to \R^{d\times d'}$ are continuous in $u\in U,$  and in $(t,x)\in[0,T]\times\R^d$ uniformly with respect to $u$. 
		
		\item There exists a measurable function $\vartheta_1:[0,T]\times U\to [0,+\infty]$
		and a constant  $c_{\rm lin}>0$ such that
		\begin{equation}\label{dar-2.1}
		\abs{b(t,x,u)}+\abs{\sigma(t,x,u)}\le c_{\rm lin}|x|+\vartheta_1(t,u),\quad (t,x,u)\in[0,T]\times \R^d\times U.
	\end{equation}

	\end{enumerate}
\end{Assumption}

The following result extends the a-priori estimates of Lemma 3.1 of  \cite{jaber2019affine} to the case of CSVEs. 
\begin{theorem}\label{a-priori}
Suppose that Assumption \ref{Assum1} holds and that $K\in L^{r}_{\rm loc}(\R_+;\R^{d\times d})$ for some $r>2.$ Let $(u_t)_{t\in[0,T]}$ be a $U$-valued adapted control process such that
\[
\Exp\int_0^T\vartheta_1(t,u_t)^p \,dt<\infty
\] for some $p$ satisfying $\frac{1}{p}+\frac{1}{r}<\frac{1}{2}.$ Let $X$ be a $\R^d$-valued solution to the controlled equation (\ref{EVC})
with initial condition $x_0\in\mathcal C(0,T;\R^d).$ Then,
	\begin{equation}\label{Est-Lp}
	\sup_{t\in [0,T]}\Exp\left[\abs{X_t}^{m} \right]\leq c,
	\end{equation}
for all $m> 2$ satisfying $\frac{1}{m} \in\bigl[\frac{1}{p}+\frac{1}{r},\frac{1}{2}\bigr),$ where the constant $c$ depends on $m,p,c_{\rm lin},T,C_B\footnote{$C_B$ is the constant in the Burkholder-Davis-Gundy inequality, see e.g. Section 4, Chapter IV in \cite{protter2005stochastic}.},$ $\abs{x_0}_{\mathcal C(0,T;\R^d)},$ $\Exp\int_0^T\vartheta_1(t,u_t)^{p}\,dt$ and  $L^ {2}-$continuously on $K|_{[0,T]}.$
\end{theorem}

\begin{proof} For simplicity, but without loss of generality, we take $d=d'=1.$ Let $t\in[0,T]$ be fixed. Then, for any $m>1$ we have
\begin{align*}
\abs{X_t}^{m}&\le 3^{m-1}\left[\abs{x_0}^{m}+\abs{\int_0^{t}K(t-s)b(s,X_s,u_s )\,ds}^{m}+\abs{\int_0^{t}K(t-s)\sigma(s,X_s,u_s)\,dW_{s}}^{m}\right]\\
&= 3^{m-1}\left[\abs{x_0}^{m}+I+II\right].
\end{align*}
Using Burkholder-Davis-Gundy inequality, and Jensen's inequality with the measure
\[
\rho(ds):=\frac{K(t-s)^2\,ds}{\int_0^t \abs{K(t-\tau)}^{2} \, d\tau}
\]
we have
\begin{align*}
\Exp[II]&\le C_{B}   \e\left[\abs{\int_0^{t}K(t-s)^2 \sigma(s,X_s,u_s )^2\,ds}^{m/2}\right]\\
&\le C_{B}\norm{K}^{m-2}_{L^{2}} \int_0^{t}\abs{\sigma(s,X_s,u_s)}^{m}\abs{K(t-s)}^{2} \, ds.
\end{align*}
By condition (\ref{dar-2.1})
\begin{align}\label{ineq2}
\e[II]&\leq C_B 2^{m-1}c_{\rm lin}^m \norm{K}^{m-2}_{L^{2}} \left(  \int_0^t \e\abs{X_s}^mK(t-s)^2  \ ds + c_{\rm lin}^{-m}\e \int_0^t  \vartheta_1(s,u_s)^m K(t-s)^2\ ds  \right)\nonumber \\ 
&=k_1\left(\int_0^{t}\e\abs{X_s}^{m}\abs{K(t-s)}^2ds+ k_2 \right).
\end{align}
Note that $k_2$ is finite since by H\"older's inequality we have
\[
k_2=c_{\rm lin}^{-m}\e \int_0^t  \vartheta_1(s,u_s)^m K(t-s)^2\, ds\le c_{\rm lin}^{-m}T^{1-\frac{m}{p}-\frac{2}{r}} \left[\e\int_0^T \vartheta_1(s,u_s)^{ p } \, ds\right]^{m / p}\norm{K}_{L^r}^{2}.
\]
A similar argument for the first term $I$ yields
\begin{align}\label{ineq1}
\e[I]&\leq t^{m/2} 2^{m-1}c_{\rm lin}^m \norm{K}^{m-2}_{L^{2}} \left(  \int_0^t \e\abs{X_s}^mK(t-s)^2  \ ds + c_{\rm lin}^{-m}\e \int_0^t  \vartheta_1(s,u_s)^m K(t-s)^2\ ds  \right)\nonumber\\
&=t^{m/2}C_{B}^{-1}k_1\left(\int_0^{t}\e\abs{X_s}^{m}\abs{K(t-s)}^2ds+ k_2 \right).
\end{align}
For each $n\in\mathds N$ set $\tau_{n}=\inf \left\{t\geq 0: \abs{X_t}\geq n\right\} \wedge T.$  By the Corollary of Theorem II.18 in \cite{protter2005stochastic} we have that,
\[
\abs{X_t}^{m}\mathbbm{1}_{\left\{t<\tau_{n}\right\}} \le\abs{x_0  +\int_0^{t}K(t-s)(b(s,X_s\mathbbm{1}_{\left\{s<\tau_{n}\right\}},u_s )\,ds)+ \sigma(s,X_s\mathbbm{1}_{\left\{s<\tau_{n}\right\} },u_s )\,dW_s) }^{m}.
\]
Let $f_{n}(t)=\e\abs{X_t}^{m} \mathbbm{1}_{\left\{t<\tau_n\right\}}.$ Then, by  (\ref{ineq1}) and (\ref{ineq2}) we have
\[
f_{n}\leq \bar{k}k_2+\bar{k}\abs{K}^{2}*f_{n}
\]
where $\bar{k}= k_1(1+T^{m/2}C_{B}^{-1}).$ Using the same argument in the proof of Lemma 3.1 of  \cite{jaber2019affine}, this yields \eqref{Est-Lp} with a constant that only depends on $m, p, c_{\rm lin},C_B,$ $T,\abs{x_0}_{\mathcal C(0,T;\R^d)}$ and   $L^ {m}-$continuously, on $K|_{[0,T]}.$ 
\end{proof}
\begin{corollary}\label{S-H}
Under the same Assumptions of Theorem \ref{a-priori}, suppose further Assumption \ref{Convcond} also holds
with $\gamma$ satisfying $\gamma>2/m$, where $\frac{1}{m} = \frac{1}{p}+\frac{1}{r}$. Then  $X$ admits a version with paths in $\mathcal C^\alpha(0,T;\R^d)$ for any  $\alpha\in\bigl[0,\frac{\gamma}{2}-\frac{1}{m}\bigr)$. For this version, denoted again with $X,$ we have the following:
\begin{equation}
\label{Holder C2}
\e\left[\abs{X-x_0}_{\mathcal C^\alpha(0,T;\R^d)}^{m}\right] \leq c,
\end{equation}
with $c$ depending on $m,p,c_{\rm lin},T,C_B,\abs{x_0}_{\mathcal C(0,T;\R^d)}, \Exp\int_0^T \vartheta_1(t,u_t)^{p}\,dt$ and  $L^ {2}-$continuously on $K|_{[0,T]}.$
\end{corollary}

\begin{proof} Follows directly from the estimate (\ref{Est-Lp}) and Lemma 2.4 in \cite{jaber2019affine}. 
\end{proof}

In particular, one can prove the following existence result for solutions to the CSVE \eqref{EVC}.
\begin{corollary} \label{Strong-S}
Let u be a $U$-valued $\mathbb{F}$-predictable process. Assume that $K$ satisfies Assumption \ref{Convcond}, $b$ and $\sigma$ satisfy Assumption \ref{Assum1} and they are Lipschitz uniformly with respect to $(t,u) \in [0,T]\times U$,  and 
\[
\Exp\int_0^T \vartheta_1(t,u_t)^p \,dt<\infty
\] for some $p$ satisfying $\frac{1}{p}+\frac{1}{r}<\frac{1}{2}.$ Suppose further that $\gamma>2\left(\frac{1}{p}+\frac{1}{r}\right)$. Then there exists a unique continuous solution X to the CSVE \eqref{EVC}.
\end{corollary}
\begin{proof} 
Using Theorem \ref{a-priori} and Corollary \ref{S-H}, the proof is completely analogous to the proof of Theorem 3.3 in \cite{jaber2019affine}.
\end{proof}

We will also frequently use the following result in the proof of the main existence result of relaxed controls. This alternative formulation of stochastic Volterra equations, by considering the integrated process $\int_0^{\cdot}X_s\, ds$, is inspired by the martingale problem approach in \cite{abi2019weak} and facilitates the justification of convergence arguments that will be useful in our setting.

\begin{lemma}\label{equiv-Ev}
Suppose that Assumption \ref{Assum1} holds, $K\in L^{2}_{\rm loc}(\R_+;\R^{d\times d})$ and
\begin{equation*}
\e\left[\int_0^T \vartheta_1^2(t,u_t) dt \right]< \infty.
\end{equation*}
Let $X$ be a solution to the CSVE (\ref{EVC}) and let $Z$ be the controlled process $Z_t=\int_0^t b(s,X_s,u_s)\,ds +\int_0^t \sigma(s,X_s,u_s)\,dW_s$. If $X$ has paths in $L^2_{\rm loc}$ then
\begin{equation}
 \label{Cresol1T}
\int _0^tX_s\,ds=\int_0^t x_0(s)\,ds + \int_0^t K(t-s)Z_s\,ds, \ \ t \in[0,T].
 \end{equation}
Conversely, if $X$ satisfies (\ref{Cresol1T}) with paths in $L^2_{\rm loc}$ then it solves the CSVE (\ref{EVC}).
\end{lemma}
\begin{proof} 
Follows from Lemma 3.2 of \cite{abi2019weak}. \end{proof}
%

\section{Relaxed control formulation}\label{sec:3}

The use of stochastic relaxed controls is inspired by the works of \cite{ElKaroui} and \cite{haussmann}. \textcolor{black}{ In what follows, $\mathcal{P}(U)$ denotes the set of all  probability measures on $\B(U)$ endowed with the $\sigma-$algebra generated by the projection maps
\begin{align*}
    \theta_C:\mathcal{P}(U)&\mapsto \pi(C)\\
     \pi &\mapsto [0,1],  \ \ C\in\B(U).
\end{align*}}
We associate a relaxed control system to the original control problem (\ref{P0})-(\ref{EVC}) as follows. First,  we extend the definition of coefficients and cost functionals with the convention
\[
\bar{F}(t,x,\pi)=\int_{U}F(t,x,u)\,\pi(du)
\]
provided that for each $t\in[0,T]$ and $x\in\R^d$ the map $F(t,x,\cdot)$ is integrable with respect to $\pi\in \mathcal P (U).$
\textcolor{black}{
\begin{definition}
	A stochastic process $\pi=(\pi_t)_{t\in [0,T]}$ with values in $\mathcal{P}(U)$ is called a \emph{stochastic relaxed control} (or relaxed control process) on $U$ if the map
	\begin{align*}
	    [0,T]\times\Omega &\mapsto \mathcal{P}(U) \\
     (t,\omega) &\mapsto \pi_t(\omega,\cdot)
	\end{align*}
	is predictable. In other words, a stochastic relaxed control on $U$ is a predictable process with values in $\mathcal{P}(U)$.
\end{definition}
}

Given a relaxed control process $(\pi_t)_{t\in [0,T]},$ the associated relaxed controlled equation now reads 
\begin{equation}
\label{Voltrm}
X_t=x_0(t)+\int_0^t K(t-s)\bar{b}(s,X_s,\pi_s)ds+\int_0^t K(t-s) \bar\sigma(s,X_s,\pi_s)\,dW_s, \ t\in [0,T],
\end{equation}
where $\bar\sigma$ is defined, with a slight abuse of notation, so that the following holds:
\[
\bar\sigma\bar\sigma^\top(t,x,\pi)=\int_U\sigma\sigma^\top(t,x,u)\,\pi(du),\quad t\in[0,T],\, x\in\R^d,\, \pi\in\calP(U).
\]
For the existence of $\bar\sigma$ see e.g, Theorem 2.5-a in  \cite{ElKaroui}. The relaxed cost functional is defined as
\[
\mathcal {J} (X,\pi)=\Exp\left[\int_0^T\bar l(t,X_t,\pi_t)\,dt+G\left(X_T\right)\right].
\]
Notice that the original system (\ref{P0})-(\ref{EVC}) controlled by a $U$-valued process $u=(u_t)_{t\in [0,T]}$ coincides with the relaxed system controlled by the Dirac measures $\pi_t=\delta_{u_t},$ $t\in[0,T]$. Moreover, since relaxed controls are just usual (strict) controls with control set $\calP(U),$ the results for strict controls in the previous section also hold for relaxed controls, with the control system defined in terms of the relaxed versions of coefficients, running cost and $\bar\vartheta_1(t,\pi).$

\subsection{Weak formulation of optimal control problem}
We study the existence of an optimal control for the stochastic relaxed control system in the following weak formulation.
\begin{definition}\label{def:relaxed_control}
	Let $T>0$ and $x_0\in\mathcal C(0,T;\R^d)$ be fixed. A weak admissible relaxed control for $(K,b,\sigma)$ is a system
	\begin{equation}
	\label{pi}
	\Theta=\left(\Omega,\mathcal F, \mathbf P, \mathbb F, W,X,\pi\right)
	\end{equation}
	such that the following hold:
	\begin{enumerate}
		\item $\left(\Omega,\mathcal F, \mathbf P \right)$ is a complete probability space endowed with a filtration $\mathds F=( \mathcal F_t)_{t\in[0,T]},$ satisfying the usual conditions,
		
		\item $W=(W_t)_{t\in[0,T]}$ is a $m$-dimensional Brownian motion with respect to $\mathds F,$
		
		\item $\pi=(\pi_t)_{t\in[0,T]}$ is a $\Fil$-predictable process with values in $\mathcal P(U)$, 
		
		\item $X=(X_t)_{t\in[0,T]}$ is a $\Fil$-adapted solution to the relaxed controlled equation (\ref{Voltrm}).
		
		\item The map $[0,T]\times \Omega\ni		(t,\omega)\mapsto \bar l(t,X_t(\omega),\pi_t(\omega))\in \R$
		belongs to $L^1([0,T]\times \Omega;\R)$ and $G(X_T)\in L^1(\Omega;\R)$.

	\end{enumerate}
\end{definition}

The set of weak admissible relaxed control systems with time horizon $[0,T]$ and initial value $x_0$ will be denoted by $\bar{\U}(x_0,T)$. Under this weak formulation, the relaxed cost functional is defined as
\begin{equation}\label{barJ}
\bar{\mathcal{J}}({\Theta})=\Exp^\Prob\left[\int_0^T\bar l(s,X_s^{\Theta},\pi_s^\Theta)\,ds+G\left(X_T^{\Theta}\right)\right], \ \ \ \Theta\in\bar{\U}(x_0,T).
\end{equation}
The relaxed control problem {\rm \textbf{(RCP)}} consists in minimizing $\bar{\mathcal{J}}$ over $\bar{\U}(x_0,T).$ Namely, we seek $\widetilde{\Theta}\in\bar{\U}(x_0,T)$ such that
\begin{equation}\label{barJ-min}
\bar{\mathcal{J}}(\tilde{\Theta})=\inf_{{\Theta}\in\bar{\U}(x_0,T)}\bar{\mathcal {J}}({\Theta}).
\end{equation}

\subsection{Main existence result}
In order to complete the set of assumptions for the main existence result, we need the following definition.
\textcolor{black}{
\begin{definition}\label{infcompact}
	A function $\vartheta:U\to[0,+\infty]$ is called \emph{inf-compact} if for every  $R\geq 0$ the level set $\set{\vartheta\le R}=\set{u\in U:\vartheta(u)\le R}$ is compact.
\end{definition}
Observe that, since $U$ is Hausdorff, for every inf-compact function $\vartheta$ the level sets $\set{\vartheta\le R}$ are closed. Therefore, every inf-compact function is lower semi-continuous and hence Borel-measurable. If $U$ is compact, the converse holds too, i.e. every lower semi-continuous function is inf-compact.
We will denote by $IC(0,T;U)$ the class of measurable functions $\vartheta:[0,T]\times U\to[0,+\infty]$ such that for all $t\in [0,T]$ the map   $\vartheta(t,\cdot)$ is inf-compact.}

\begin{Assumption}\label{Assum2}
	\begin{enumerate}
		\item The \textbf{control set} $U$ is a metrizable \textbf{Suslin} space i.e. there exists a Polish space $S$ and a continuous mapping $\phi:S\to U$ such that $\phi(S)=U$.
		
			\item The \textbf{running cost} function $l:[0,T]\times \R^d\times U\to (-\infty,+\infty]$ is measurable in $t\in [0,T]$ and lower semi-continuous with respect to $(x,u)\in \R^d\times U$.
		
		\item There exist
		$\vartheta_2\in IC(0,T;U)$ and
		constants $C_1\in\R, \ C_2>0$  such that $h$ satisfies the  following \textbf{coercivity} condition:
		\begin{equation}\label{ineq-coercive}
		\vartheta_2(t,u)^p\le C_1+C_2 l(t,x,u), \ \ \ (t,x,u)\in [0,T]\times \R^d\times U
		\end{equation}
		for some $p\geq 1.$
		
		\item The \textbf{final cost} function $G:\R^d\to\R$ is continuous.
	\end{enumerate}
\end{Assumption}

The following is the main result of this paper.
\begin{theorem}[Existence of optimal relaxed controls]
	\label{thm-main}
	Let $T>0$ and $x_0\in\mathcal C(0,T;\R^d)$ be fixed. Suppose that Assumptions \ref{Convcond}, \ref{Assum1} and \ref{Assum2} hold with $r>2,$ $\gamma\in(0,2]$ and $p$ satisfying $\frac{1}{p}+\frac{1}{r}<\frac12$ and $\gamma>2\left(\frac{1}{p}+\frac{1}{r}\right)$. Suppose further that $\vartheta_1 \leq  \vartheta_2$ and there exists $\Theta \in\bar{{\U}}(x_0,T)$ such that
	$\bar{\mathcal{J}}({\Theta})<+\infty,$ then \textbf{\emph{(RCP)}} admits a weak optimal relaxed control.
\end{theorem}

\begin{example}[Fractional kernel]\label{Ex-Frac}
For simplicity, we fix $d=d'=1,$ and consider the fractional kernel $K(t)=t^{H-\frac12}$ with \textcolor{black}{$H\in\bigl(\frac{1}{4},\frac{1}{2}\bigr).$} Suppose the coefficients and running cost function have the form
\begin{align*}
b(t,x,u)&=b_0(t,u)+b_1(t,u)x\\
\sigma(t,x,u)&=\sigma_0(t,u)+\sigma_1(t,u)x\\
l(t,x,u)&=l_0(t,u)+l_1(x)
\end{align*}
with
\begin{itemize}
	\item[\textbullet] $b_i,\sigma_i$ measurable and continuous in $u\in U,$ uniformly with respect to $t\in [0,T],$ for $i=0,1,$
	\item[\textbullet] $b_1,\sigma_1$ uniformly bounded in $(t,u),$
	\item[\textbullet] $l_0\in IC(0,T;M)$ and $l_1$ LSC and bounded from below.
\end{itemize}
Suppose further that $\abs{f(t,u)}^p\le C l_0(t,u)$ for both $f=b_0,\sigma_0$, some constant $C>0$ and $p$ sufficiently large satisfying
$\frac{1}{p}<2H-\frac{1}{2}.$ Then, there exists $r>2$ such that
\[
\frac12-H<\frac{1}{r}<H-\frac{1}{p}
\]
so that Assumption \ref{Convcond} holds for this choice of $r$ and $\gamma=2H.$ Assumptions \ref{Assum1}, \ref{Assum2} hold with $\vartheta_1=\vartheta_2=(Cl_0)^{1/p}$, $C_1=-C\inf l_1$ and $C_2=C.$ 
Then, existence of an optimal relaxed control follows from Theorem \ref{thm-main}.
\end{example}
\begin{remark}
Recently, \cite{jaber2019linear} proved an existence result for Linear-Quadratic control problems for linear Volterra equations, and obtained a linear feedback characterization of optimal controls. Unlike \cite{jaber2019linear}, we do not assume linearity in the coefficients with respect to the control variable. Our assumptions, however, do not cover cost functions with `quadratic growth' in the control variable, since we are forced to choose $p$ strictly larger than 2.
\end{remark}

\subsection{Existence of strict controls}
Our main result on the existence of optimal strict controls requires one additional assumption, familiar in relaxed control theory since the work of \cite{filippov}, to ensure existence of an optimal strict control. 

\begin{Assumption}\label{A3}
\begin{enumerate}
\item $U$ is a closed subset of a Euclidean space.

\item For each $(t,x)\in[0,T]\times\R^d,$ the set 
\begin{equation}\label{eq:Gamma}
\Gamma(t,x)=\left\{ \bigl(\sigma \sigma^{\top}(t,x,u),b(t,x,u),z\bigr): u\in U, \, z \geq l(t,x,u)\right\}
\end{equation}
 is a convex and closed subset of $\mathcal S^{d}\times \R^{d}\times \R.$
\end{enumerate}
 \end{Assumption}
\begin{theorem}[Existence of optimal strict controls]\label{thm-strict}
Suppose that Assumption \ref{A3} holds. Then, for each $\Theta=\left(\Omega,\mathcal F, \mathbf P, \mathds F, W,X,\pi\right) \in \bar{\U}(x_0,T)$ there exists a $U$-valued $\Fil$-predictable control  process $u=(u_t)_{t\in [0,T]}$ on the same probability space $(\Omega,\mathcal F, \mathbf P)$  such that
\begin{enumerate}
	\item $X$ satisfies the Volterra equation (\ref{EVC}) controlled by the strict control process $u=(u_t)_{t\in [0,T]}.$
	
	\item $\int_0^T\bar{l}(s,X_s,\pi_s)\,ds \geq \int_0^T l(s,X_s,u_s)\,ds,$ $\Prob$-a.s. 
\end{enumerate}
In particular, if the Assumptions of Theorem \ref{thm-main} also hold, there exists a weak optimal strict control for (\ref{P0})-(\ref{EVC}).
\end{theorem}

\begin{example}
Let $U=\mathbb{R}$ or $U=[-\tilde{U},\tilde{U}],$ with $0<\tilde{U}< \infty,$ $d=d'=1$ and $K(t)=t^{H-\frac12}$ with $H\in\bigl(\frac{1}{4},\frac{1}{2}\bigr).$ Suppose the coefficients  have the form
\begin{align*}
b(t,x,u)&=b_0(t,x)+b_1(t,x)u^2\\
\sigma(t,x,u)&=\sigma_1(t,x)u\\
l(t,x,u)&=l_0(t,u^2)+l_1(x),
\end{align*}
where
\begin{itemize}
	\item[\textbullet] $b_i,\sigma_1$ measurable and continuous in $x\in \R,$ uniformly respect to $t\in [0,T],$ for $i=0,1,$
	\item[\textbullet] $b_1,\sigma_1$ uniformly bounded in $(t,x)$ and $\abs{b_0(t,x)}\leq K\abs{x},$ with $K>0,$
	\item[\textbullet] $l_0(t,.)$ is a function convex on $\R^{+},$ for each $t\in [0,T],$ $l_0\in IC(0,T;U^2)$ and  $l_1$ LSC  bounded from below.
\end{itemize}
Suppose further that $\abs{\phi(t,u)}^p\le C l_0(t,u),$ with $\phi(t,u)=\max\{\abs{u}^2,\abs{u} \}$ for all $t\in[0,T],$ and $p$ satisfying $\frac{1}{p}<2H-\frac{1}{2}.$ As in the Example \ref{Ex-Frac} there exists $r>2$ such that  Assumptions \ref{Assum1}, \ref{Assum2} hold with $\vartheta_1=\vartheta_2=(Cl_0)^{1/p}$, $C_1=-C\inf l_1$ and $C_2=C.$ By Theorem \ref{thm-main} there is an optimal relaxed control. Let $\Gamma^1(t,x)=\left\{(\widetilde u,z): \widetilde u\in U^2, z\geq l_0(t,\widetilde u)+l_1(x)  \right\}.$ Then  $\Gamma$ in \eqref{eq:Gamma} can be written as an affine transformation of $\Gamma^1.$ More precisely,  $\Gamma(t,x)=\textbf{b}(t,x)+\textbf{A}(t,x)\Gamma^1(t,x)$ where
\[
\textbf{b}(t,x)=\begin{bmatrix} 0\\ b_0(t,x)\\0\end{bmatrix}, \ \ \textbf{A}(t,x) =\left[\begin{matrix} \sigma_1^2(t,x) &0\\ b_1(t,x)&0\\0&1\end{matrix}\right],  \ \ \ (t,x) \in [0,T]\times \R.
\]
Since $l_0$ is a function convex on $\R_+$  the epigraph $\Gamma^1$ is a convex set, then $\Gamma$ is a convex set and by Theorem \ref{thm-strict}  there is an optimal strict control.
\end{example}

\begin{remark}
If  $\sigma$ does not depend on $u\in U$, Assumption \ref{A3} holds if the set $\bigl\{ (b(t,x,u),z): u\in U, z \geq l(t,x,u)\bigr\}$ is convex and closed in $\R^{d}\times \R.$ This is the case, for instance, if the drift coefficient is affine in $u,$ i.e. it has the form $b(t,x,u)=b_0(t,x)+b_1(t,x)u$ and if $l(t,x,\cdot)$ is lower semi-continuous and convex.
\end{remark}

\section{Proofs of the main theorems}\label{sec:4}
\subsection{Relaxed controls and Young measures}
\begin{definition}
	Let $\mathcal L(dt)$ denote the Lebesgue measure on $[0,T]$ and $\mu$ be a bounded non-negative $\sigma-$additive measure on $\B\left(U\times [0,T]\right)$. We say that $\mu$ is a \emph{Young measure} on $U$ if and only if $\mu$ satisfies
	\begin{equation}\label{youngmeasure}
		\mu(U\times D)=\mathcal L(D),\quad D\in \B([0,T]),
	\end{equation}
	i.e. the marginal  of $\mu$ on $\B([0,T])$ is equal to the Lebesgue measure. 	We denote by $\Y(0,T;U)$ the set of Young measures on $U.$ We endow $\Y(0,T;U)$ with the \emph{stable topology} defined as the weakest topology for which the mappings 
	\[
	\Y(0,T;U)\ni\mu\mapsto\ \int_{U\times D} f(u)\,\mu(du,dt)\in\R
	\]
	are continuous, for every $D\in\B([0,T])$ and $f\in\mathcal{C}_b(U)$.
\end{definition}

The following result connects random Young measures with predictable relaxed controls. For the proof, see e.g. Section 3.3 of \cite{kushner2012weak} or Section 2.4 of \cite{cecchin2020probabilistic}. 
\begin{lemma}[\textbf{Predictable disintegration of random Young measures}]\label{Lem:disrym}
	Let $(\Omega,\calF,\Prob)$ be a probability space and let $U$ be a Radon space. Let $\mu:\Omega\to\Y(0,T;U)$ be such that, for every $J\in\B(U\times [0,T])$,  the mapping
	\[
	\Omega\ni \omega\mapsto \mu(w)(J)=\mu(\omega,J)\in [0,T]
	\]
is measurable. Then there exists a stochastic relaxed control $(\pi_t)_{t\in [0,T]}$ on $U$ such that for $\Prob-$a.e. $\omega\in\Omega$ we have
	\begin{equation}\label{disrym}
		\mu(\omega,C\times D)=\int_D \pi_t(\omega,C)\,dt,\quad C\in\B(U), \ D\in\B([0,T]).
	\end{equation}
Moreover, if   $\mathds F=( \mathcal F_t)_{t\in[0,T]}$ is a given filtration that satisfies the usual conditions and $\mu([0,\cdot)\times C)$ is $\mathds F-$adapted, for all $C\in\B(U),$ then $\pi$ is a $\mathds F-$predictable process.
\end{lemma}

\begin{remark}
We will  denote the disintegration formula (\ref{disrym}) by $\mu(du,dt)=\pi_t(du)\,dt$. Note that $\pi_t(C)$ can be seen as the time-derivative of $\mu([0,t)\times C)$ that exists for almost every $t\in[0,T],$ for all $C\in\mathcal B(U).$ 
\end{remark}

\begin{remark}
	It can be proved (see e.g. Remark 3.20  \cite{crauel}) that if $U$ is a separable and metrisable topological space, then $\mu:\Omega\to\Y(0,T;U)$ is measurable with respect to the Borel $\sigma-$algebra generated by the stable topology if and only if for every $J\in\B(U\times [0,T])$ the mapping
	\[
	\Omega\ni \omega\mapsto \mu(w)(J)\in [0,T]
	\]
	is measurable. This justifies referring to the maps considered in Lemma \ref{Lem:disrym} as random Young measures.
\end{remark}
For the two following Lemmas, $E$ denotes a Euclidean space with norm $\abs{\cdot}$ and inner product $\seq{\cdot,\cdot}$. 
\begin{lemma}\label{Sigmat-meas}
	Let $f:[0,T]\times\R^d\times U\to E$ be a Borel-measurable function, continuous in $u\in U,$ and continuous in $x\in\R^d$ uniformly with respect to $u\in U$, satisfying the growth condition
\begin{equation}\label{cond-f}
	|f(t,x,u)|_E\le c_{\rm lin}|x|^\delta+\vartheta(t,u)
\end{equation}
with $\vartheta\in IC(0,T;U),$ for some $\delta\geq 1.$ For $\beta\geq 1$ fixed, we denote
\[
\Y^\beta(0,T;U):=\left\{\mu\in\Y(0,T;U):\vartheta\in L^\beta(\mu) \right\}.
\]
Then, for each $t\in [0,T]$,  the mapping $\Sigma_t:\mathcal{C}(0,T;\R^d)\times\Y^\beta(0,T;U)\to \R^d$ defined by
	\begin{equation}\label{Sigmat}
		\Sigma_t(x,\mu)=\int_{U \times [0,t]} f(s,x(s),u)\,\mu(du,ds),
	\end{equation}
	is Borel-measurable.
\end{lemma}
\begin{proof}  We fix $t\in [0,T].$ 	For each $N\in\mathds{N}$ and $i\in\set{1,\ldots,d}$ define
	\[
	\phi^i_N(\mu)=\int_{U \times [0,t]} \min\{N,f^i(s,x(s),u)\} \,\mu(du,ds), \ \ \ \mu\in\Y^\beta(0,T;U).
	\]
	The integrand in the above expression is bounded and continuous with respect to $u\in U$. Therefore, by Theorem \ref{thm:contym} $\phi_N$ is continuous for each $N\in\mathds{N}$,  and by dominated convergence, $\phi_N^i(\mu)\to \Sigma_t^i(x,\mu)$ as $N\to\infty$ for all $\mu\in\Y^\beta(0,T;U)$. Hence, $\Sigma_t(x,\cdot)$ is measurable.
	Now, we prove that for $\mu\in\Y^\beta(0,T;U)$ fixed, the map $\Sigma_t(\cdot,\mu)$ is continuous. Let $x_n\to x$ in $\mathcal{C}\left(0,T;\R^d\right)$. Then, by assumption we have
	\[
	\bigr|f(s,x(s),u)-f(s,x_n(s),u)\bigr|\to 0 \mbox{ as }n\to\infty,\quad (s,u)\in [0,t]\times U.
	\]
	Moreover, since $(x_n)$ converges in $\mathcal{C}\left(0,T;\R^d\right)$, it is bounded and there exists $\bar\rho>0$ such that 
	\[
	\sup_{s\in [0,T]}\bigl|x_n(s)-x(s)\bigr| < \bar\rho, \ \ \ \forall n\in\mathds{N}.
	\]
	Therefore by (\ref{dar-2.1}), we have
	\begin{eqnarray*}
		&&\bigl|f(s,x(s),u)-b(s,x_n(s),u)\bigr|
		\le c_{\rm lin}\left\{\bar\rho+2\abs{x}_{\mathcal{C}(0,T;\R^d)}\right\}+\vartheta(s,u)
	\end{eqnarray*}
	As $\vartheta$ belongs to $L^1([0,T]\times U;\mu)$,  so does the right side of the above inequality. Therefore, by Lebesgue's dominated convergence theorem we have
	\begin{eqnarray*}
	\abs{\Sigma_t(x,\mu)-\Sigma_t(x_n,\mu)}
	\le \int_{U\times [0,t]}\abs{f(s,x(s),u)-f(s,x_n(s),u)}\,\mu(du,ds)\to 0
	\end{eqnarray*}
	as $n\to\infty$,  that is, $\Sigma_t(\cdot,\mu)$ is continuous. Since $\Y^\beta(0,T;U)$ is separable and metrizable, by Lemma 1.2.3 in \cite{cadfval} it follows that $\Sigma_t$ is jointly measurable. \end{proof}

Recall that $\phi^n\rightharpoonup \phi$ weakly in $L^1([0,T]\times\Omega;E)$ if 
\[
\Exp\int_0^T \seq{\phi^n(t),\psi(t)}\,dt \to  \Exp\int_0^T \seq{\phi(t),\psi(t)}\,dt,\quad \forall\psi \in L^\infty([0,T]\times\Omega;E).
\]
We have the following result. 
\begin{lemma}\label{weak-conv-coeff}
	Let $f:[0,T]\times\R^d\times U\to E$ be a Borel-measurable function, continuous in $x\in\R^d$ uniformly with respect to $u\in U$, satisfying the growth condition (\ref{cond-f}) 
	for some $\delta\geq 1,$ with $\vartheta\in IC(0,T;U).$ Let $(X^n)_{n\in\mathds N}$ a sequence of $\R^d$-valued processes, and $\mu^n(du,dt)=\pi_t^n(du)\,dt$ a sequence of stochastic relaxed controls defined on the same probability space $(\Omega,\mathcal F,\Prob)$ such that $X^n\to X$ point-wise $\Prob$-a.s. and  in $L^{\beta\delta}(\Omega\times[0,T];\R^d)$ for some $\beta>1,$ and $\mu^n\to \mu$ in the stable topology $\Prob$-a.s., with $\mu(du,dt)=\pi_t(du,dt).$ Suppose further
	\begin{equation}\label{unif-bound-eta-beta}
		\sup_{n\in\mathds N}\Exp^{\Prob}\int_{U\times[0,T]}\vartheta(t,u)^\beta\,\mu^n(du,dt)<\infty.
	\end{equation}
	For each $n\in\mathds{N}$, set $f^n_t=\bar f(t,X_t^n,\pi_t^n), \ \hat f^n_t=\bar f(t,X_t,\pi_t^n)$,  $t\in [0,T]$. Then
	\begin{enumerate}
		\item $f^n-\hat f^n\to 0,$ (strongly) in $L^1([0,T]\times\Omega;E).$
		\item $\hat f^n\rightharpoonup f,$ weakly in $L^1([0,T]\times\Omega;E).$
	\end{enumerate}
	with $f_t=\bar f(t,X_t,\pi_t),$  $t\in [0,T]$.
\end{lemma}
 \begin{proof} We first prove
	\begin{equation}\label{tfnhatfn}
		f^n-\hat f^n\to 0, \ \ \ \ \mbox{(strongly) in} \ L^1([0,T]\times\Omega;E).
	\end{equation}
	By uniform continuity with respect to $u\in U$, for each $n\in\mathds N$ we have
	\begin{eqnarray*}
		I^n_t=\int_U\bigl|f(t,X^n_t,u)-f(t,X_t,u)\bigr|\,\pi^n_t(du)\le\sup_{u\in U}\bigr|f(t,X^n_t,u)-f(t,X_t,u)\bigr|\to 0
	\end{eqnarray*}
	as $n\to\infty$ for $t\in[0,T]$,  $\Prob-$a.s. From (\ref{cond-f}) and (\ref{unif-bound-eta-beta}), we get
	\[
	\sup_{n\in\mathds{N}}\e\int_0^T\abs{I^n_t}^{\beta}\,dt<+\infty.
	\]
	Hence, $\{I^n\}_{n\in\mathds{N}}$ is uniformly integrable on $\Omega\times [0,T]$. Lemma 4.11  \cite{kallenberg} implies that
	\[
	\e\int_0^T\bigl|f^n_t-\hat f^n_t\bigr|\,dt\le\e\int_0^T\!I^n_t\,dt\to 0, \ \ \ \mbox{ as } \ n\to\infty,
	\]
	and (\ref{tfnhatfn}) follows. Now, we will prove that
	\begin{equation}\label{hatfntf}
		\hat f^n\rightharpoonup f, \ \ \ \mbox{ weakly in } \ L^1([0,T]\times\Omega; E).
	\end{equation}
	Let $\psi\in L^{\infty}([0,T]\times\Omega;E)$ be fixed. We denote $g(t,u)=\bigl\langle f(t,X_t,u),\psi_t\bigr\rangle.$ Then, 
\[
\e\int_0^T\bigl\langle \hat f^n_t,\psi_t\bigr\rangle\,dt=\e\int_0^T\seq{\int_U f(t,X_t,u)\,\pi^n_t(du),\psi_t}\,dt=\e\int_{U\times[0,T]} g(t,u)\,\mu^n(du,dt)
\]	for each $n\in\mathds{N}$. Let $\varepsilon\in(0,1)$ be fixed and take $C_\varepsilon>\max\{\frac{R}{\varepsilon},1\}$ with $R$ defined as the supremum in (\ref{unif-bound-eta-beta}), and let $\calA_{\varepsilon}=\set{(t,u)\in[0,T]\times U:\vartheta(t,u)^{\beta-1}>C_\varepsilon}.$ Then, for this choice of $C_\varepsilon$, we have
	\begin{eqnarray*}
		\e\left[\mu_n\bigl(\calA_{\varepsilon}\bigr)\right]
		=\e\int_{\calA_{\varepsilon}}\mu_n(du,dt)\le\frac{1}{C_\varepsilon}\e\int_{\calA_{\varepsilon}}\vartheta(t,u)^{\beta-1}\,\mu_n(du,dt)<\varepsilon.
	\end{eqnarray*}
We write
\[
	\e\int_{[0,T]\times U} g(t,u)\,\mu_n(du,dt)=\e\int_{\calA_{\varepsilon}^{\mathrm c}}g(t,u)\,\mu_n(du,dt)+\e\int_{\calA_{\varepsilon}}g(t,u)\,\mu_n(du,dt)
\]
	and observe first that by Theorem \ref{thm:contym} we have $\Prob-$a.s.
\[
\int_{\calA_{\varepsilon}^{\mathrm c}}g(t,u)\,\mu_n(du,dt) \to \int_{\calA_{\varepsilon}^{\mathrm c}}g(t,u)\,\mu(du,dt)
\]		
	as $n\to\infty$ and, by (\ref{cond-f}), 
\[
\int_{\calA_{\varepsilon}^{\mathrm c}}g(t,u)\,\mu_n(du,dt) \le \left[c_{\rm lin}\abs{X}^\delta_{L^1(0,T;\mathbb{R}^d)}+C_\varepsilon^{1/(\beta-1)}\right]\abs{\psi}_{L^{\infty}(0,T;E)}, \ \ \Prob-\mbox{a.s.}
\]
The right side of the last inequality has finite expectation by the hypothesis about $\psi$ and the Cauchy-Schwarz' inequality. Thus, using Lebesgue's dominated convergence theorem we get
\[	\Exp\int_{\calA_{\varepsilon}^{\mathrm c}}g(t,u)\,\mu_n(du,dt) \to \Exp\int_{\calA_{\varepsilon}^{\mathrm c}}g(t,u)\,\mu(du,dt)
\]
as $n\to\infty$. Now, for each $n\in\mathds{N}$,  define the measure $\kappa_n(du,dt,d\omega)=\mu_n(\omega)(du,dt)\Prob(d\omega)$ on $\B(U)\otimes\B([0,T])\otimes\Fil,$ 
	so we have
\[
\e \int_{\calA_{\varepsilon}}\Bigl|g(t,u)\Bigr|\,\mu_n(du,dt)\le\int_{\Omega\times\calA_{\varepsilon}}\varphi(t)\,\kappa_n(du,dt,d\omega)+\int_{\Omega\times\calA_{\varepsilon}}\vartheta(t,u)\abs{\psi_t}_{E}\,\kappa_n(du,dt,d\omega)
\]
with $\varphi=c_{\rm lin}\bigl|X\bigr|^{\delta}\abs{\psi}_{E}\in L^{\beta}([0,T]\times\Omega),$  since $\bigl|X\bigr|^\delta\in L^\beta([0,T]\times\Omega)$ and $\psi \in L^\infty([0,T]\times\Omega)$. Using  H\"{o}lder's inequality we get
	\begin{align*}
\int_{\Omega\times\calA_{\varepsilon}}\varphi(t)\,\kappa_n(du,dt,d\omega)
		\le &\left[\int_{\Omega\times[0,T]\times U}\varphi(t)^\beta\,\kappa_n(du,dt,d\omega)\right]^{1/\beta}
		\cdot\left(\e\left[\mu_n\bigl(\calA_{\varepsilon}\bigr)\right]\right)^{1-1/\beta}\\
		<&\norm{\varphi}_{L^\beta([0,T]\times\Omega)}\varepsilon^{1-1/\beta}
	\end{align*}
	and
	\begin{align*}		\int_{\Omega\times\calA_{\varepsilon}}\vartheta(t,u)\abs{\psi(t)}_{E}\,\kappa_n(du,dt,d\omega)\le&\norm{\psi}_{L^\infty([0,T]\times\Omega;E)}
		\int_{\Omega\times\calA_\varepsilon}\vartheta(t,u)\,\kappa_n(du,dt,d\omega)\\
		=&\norm{\psi}_{L^\infty([0,T]\times\Omega;E)}
		\e\int_{\calA_\varepsilon}\frac{\vartheta(t,u)^\beta}{\vartheta(t,u)^{\beta-1}}\,\mu_n(du,dt)\\
		\le&\norm{\psi}_{L^\infty([0,T]\times\Omega;E)}
		\frac{1}{C_\varepsilon}\e\int_{\calA_\varepsilon}\vartheta(t,u)^\beta\,\mu_n(du,dt)\\
		\le& \norm{\psi}_{L^\infty([0,T]\times\Omega;E)}\frac{R}{C_\varepsilon}\\
		<&\norm{\psi}_{L^\infty([0,T]\times\Omega;E)}\varepsilon
	\end{align*}
	which holds uniformly with respect to $n\in\mathds{N}$. Since $\vartheta(t,\cdot)$ is lower semi-continuous for all $t\in [0,T]$,  by Lemma \ref{Lem:lscym} and Fatou's lemma we have
	\[	\e\int_{U\times[0,T]}\vartheta(t,u)^\beta\,\mu(du,dt)\le \liminf_{n\to \infty} \e\int_{U\times[0,T]}\vartheta(t,u)^\beta\,\mu_n(du,dt)\le R.
 \]
	Therefore, the same estimates hold for $\mu$,  that is,
	\begin{eqnarray*}
		\e \int_{\calA_\varepsilon}\bigl|g(t,u)\bigr|\,\mu(du,dt) \le \norm{\varphi}_{L^\beta([0,T]\times\Omega)}\varepsilon^{1-1/\beta}+\norm{\psi}_{L^\infty([0,T]\times\Omega;E)}\varepsilon
	\end{eqnarray*}
	and since $\varepsilon\in(0,1)$ is arbitrary, we conclude that
	\[
	\e\int_{U\times[0,T]}g(t,u)\,\mu_n(du,dt)\to \e\int_{U\times[0,T]}g(t,u)\,\mu(du,dt)
	\]
	as $n\to\infty$, and (\ref{hatfntf}) follows. \end{proof}

\subsection{Proof of Theorem \ref{thm-main}}
Let $\Theta^n=\left(\Omega^n,\Fil^n,\Prob^n,W^n,\pi^n,X^n\right),$ $n\in\mathds{N}$ be a minimizing sequence of weak admissible relaxed controls, that is,
\[
\lim_{n\to\infty}\bar{\mathcal{J}}(\Theta^n)=\inf_{\theta\in\bar{\U}(x_0,T)}\bar {\mathcal{J}}({\Theta}).
\]
From this and Assumption \ref{Assum2} it follows that there exists $R>0$ such that for all $n\in\mathds{N}$
\begin{equation}
\Exp^n\int_{U\times[0,T]}\vartheta_2(t,u)^{p}\,\pi_t^n(du)\,dt\leq C_1+C_2\Exp^n\int_{U\times[0,T]} l(t,X_n(t),u)\,\pi_t^n(du)\,dt\leq R\label{estq}
\end{equation}
where $\Exp^n$ denotes expectation with respect to $\Prob^n.$ We will divide the proof in several steps.

\noindent\textsc{Step 1.} Define $m$ by $\frac{1}{m}=\frac{1}{p}+\frac{1}{r}$ and let $\alpha \in \bigl[0,\frac{\gamma}{2}-\frac{1}{m} \bigr)$ be fixed. By Corollary \ref{S-H} and \eqref{estq} the processes $X^n$ admit versions, which we also denote with $X^n,$ with paths in $\mathcal C^{\alpha}(0,T;\R^d)$ satisfying
\[
\sup_{n\in\mathds N}\Exp^n\left[\abs{X^n-x_0}_{\mathcal C^{\alpha}(0,T;\R^d)}^m\right]<\infty.
\]
Since $C^{\alpha}(0,T;\R^d)$ is compactly embedded in $\mathcal C(0,T;\R^d),$ by Chebyshev's inequality it follows that the family of laws of $\left\{X^n\right\}_{n\in\mathds N}$ is tight in $\mathcal C(0,T;\R^d).$
Using Lemma \ref{equiv-Ev}, for each $n\in \mathds{N}$ the process $X^n$ satisfies
\begin{equation*}
\int _0^tX^n_s\,ds=\int_0^t x_0(s) ds + \int_0^t K(t-s)\left( \zeta^n_s+ Y^n_s\right)ds, \ \ \ t\in[0,T],
\end{equation*}
where
\[
\zeta^n_t=\int_0^t\bar{b}(s,X^n_s,\pi^n_s)\,ds  \quad \mbox{and} \quad Y^n_t=\int_0^t\bar\sigma(s,X^n_s,\pi^n_s)\ dW^n_s.
\]
A similar argument as in the proof of Theorem \ref{a-priori} and Corollary \ref{S-H} with $K$ replaced by the identity matrix of size $d$ ensures that $\left\{\zeta^n\right\}_{n\in \mathds{N}}$ and $\left\{Y^n\right\}_{n\in \mathds{N}}$ are also tight in $C([0,T],\R^{d}).$ For each $n\in \mathds N$ we define the random Young measure
\begin{equation}
    \label{Ymesu}
    \mu_n(du,dt)=\pi_t^n(du)\,dt.
\end{equation}
We also claim that the family of laws of $\left\{\mu_n \right\}_{n\in\mathbb N}$ is tight in $\Y(0,T;U).$ Indeed, for each $\varepsilon>0$ define the set
\[
K_\varepsilon=\left\{\mu\in\Y(0,T;U):\int_{U\times[0,T]}\vartheta_2(t,u)^p\,\mu(du,dt)\le\frac{R}{\varepsilon}\right\}.
\]
By Theorems \ref{flexiblytight} and \ref{prohym}, $K_\varepsilon$ is relatively compact in the stable topology of $\Y(0,T;U)$, and by Chebyshev's inequality we have
\[
\Prob^n\left(\mu_n\in\Y(0,T;U)\setminus\bar{K_\varepsilon}\right)\le \Prob^n\left(\mu_n\in\Y(0,T;U)\setminus K_\varepsilon\right)\le \frac{\varepsilon}{R}\, \Exp^{n}\int_{U\times[0,T]}\vartheta_2^p(t,u)\,\mu_n(du,dt)
\le\varepsilon
\]
and the tightness of the laws of $\{\mu_n\}_{n\in\mathds N}$ follows. We use now Prohorov's theorem to ensure existence of a probability measure $\Psi$ on $C\left([0,T],\R^{d} \right)^3\times\Y(0,T;U)$ and a subsequence of $\left\{X^n,\zeta^n,Y^n, \mu_n \right\}_{n\in\mathbb N},$ which we denote using the same index $n\in \mathbb N,$ such that
\begin{equation}
    \label{SubsPr}
    \mathrm{law}\left(X^n,\zeta^n,Y^n,\mu_n \right) \rightarrow \Psi, \quad n\rightarrow \infty.
\end{equation}
\noindent\textsc{Step 2.} Dudley's generalization of the Skorohod representation theorem (see Theorem 4.30 in  \cite{kallenberg}) ensures existence of a probability space $(\tOmega,\tFil,\tProb)$ and a sequence of random variables $\{\tX^n,\tilde\zeta^n,\tilde Y^n,\tilde{\mu}_n\}_{n\in\mathds{N}}$ with values in $\mathcal{C}([0,T];\R^{d})^3\times\Y(0,T;U)$,  defined on $(\tOmega,\tFil,\tProb)$,  such that
\begin{equation}\label{eqlawsn}
(\tX^n,\tv^n,\tilde Y^n,\tilde {\mu}_n)\stackrel{d}{=}(X^n,\zeta^n,Y^n,\mu_n) , \quad n\in\mathds{N},
\end{equation}
and, on the same stochastic basis $(\tOmega,\tFil,\tProb)$,  a random variable $\bigl(\tX,\tv,\tilde Y,\tilde{\mu}\bigr)$ with values in $\mathcal{C}([0,T];\R^d)^3\times\Y(0,T;U)$ such that
\begin{equation}\label{tXn}
(\tX_n,\tv_n,\tilde Y_n)\to(\tX,\tv,\tilde Y), \ \ \ \mbox{in } \ \mathcal{C}([0,T];\mathbb {R}^d)^3, \ \ \tProb-\mbox{a.s.}
\end{equation}
and
\begin{equation}\label{tlambdan}
\tilde{\mu}_n\to \tilde{\mu}, \ \ \ \mbox{stably in } \ \Y(0,T;U), \ \ \tProb-\mbox{a.s.}
\end{equation}
\textsc{Step 3.} For each $t\in[0,T]$ let $\varphi_t$ denote the evaluation map $\mathcal C([0,T];\R^d)\ni\zeta\mapsto \zeta(t)\in\R^d,$ and let $\Gamma_t: C([0,T];\R^d)^2\times \Y^{p}(0,T;U)\rightarrow \R^d $ be defined as
\[
\Gamma_t(x,\zeta,\mu)=\Sigma_t(x,\mu)-\varphi_t(\zeta), \ (x,\zeta)\in C([0,T];\R^d)^2, \ \mu \in \Y^{p}(0,T;U)
\]
with $\Sigma_t$ as in (\ref{Sigmat}) with $f=b.$ Using Lemma \ref{Sigmat-meas} with $\vartheta=\vartheta_1,$ it follows that $\Gamma_t$ is measurable. Hence by (\ref{eqlawsn}) and the definition of $\zeta^n$, for each $t\in[0,T]$ and $n\in\mathds N$ we have
\begin{equation*}
\tilde\zeta^n_t=\int_{U\times[0,t]} b(s,\tilde X_s^n,u)\,\tilde\mu^n(du,ds)
\end{equation*}
By Theorem 6.1 in  \cite{gripenberg1990volterra}, the map $Z\mapsto \int_0^t K(t-s)Z_s\,ds$ is continuous from $\mathcal C(0,T;\R^d)$ to itself. In particular, it is measurable, so we also have
\begin{equation}\label{mart2}
\int _0^t \tilde X^n_s\,ds=\int_0^t x_0(s)\,ds + \int_0^t K(t-s)\left[\tilde\zeta^n_s+ \tilde Y^n_s\right]\,ds.
\end{equation}
Since $U$ is a Suslin space, it also separable and Radon, see e.g. Ch. II in \cite{schwartz}. In particular, Lemma \ref{Lem:disrym} applies, so there exists a relaxed control process $(\tilde{\pi}_t^n)_{t\in [0,T]}$ defined on $(\tOmega,\tFil,\tProb)$ such that
\begin{equation*}
\tilde{\mu}^n(du,dt)=\tilde{\pi}_t^n(du)\,dt, \ \ \ \tProb-\mbox{a.s.}
\end{equation*}
Now,  $Y^n$ is a $\Fil^n$-martingale with quadratic variation
\[
\seq{Y^n}_t=\int_0^t(\bar\sigma\bar\sigma^\top)(s,X^n_s,\pi_s)\,ds, \ \ t\in[0,T]
\]
and $(X_n,\pi^n)\stackrel{d}{=}(\tX_n,\tilde{\pi}^n).$ Then, using once again Lemma \ref{Sigmat-meas}, now with $f=\sigma\sigma^\top$ and $\vartheta=\vartheta_1^2,$ it follows that $\tilde Y^n$ is also a martingale with respect to the filtration
\[
 \tilde{\mathcal F}^n_t =\sigma\left\{ (\tX^{n}_s,\tilde{\pi}^n_s): s\in[0,t]\right\}, \ \ t\in[0,T]
 \]
and  quadratic variation $\bigl\langle\tilde Y^n\bigr\rangle_t=\int_0^t(\bar\sigma\bar\sigma^\top)(s,\tilde X^n_s,\tilde \pi^n_s)\,ds.$
Again, using continuity of the map  $Z\mapsto \int_0^t K(t-s)Z_s\,ds$ from $\mathcal C(0,T;\R^d)$ to itself, we obtain
\begin{equation*}
\int _0^t \tilde X_s\, ds=\int_0^t x_0(s)\, ds + \int_0^t K(t-s)\left[\tilde\zeta_s+ \tilde Y_s\right]\,ds.
\end{equation*}
We use Lemma \ref{Lem:disrym} one last time to ensure the existence of a relaxed control process $(\tilde{\pi}_t)_{t\in[0,T]}$ 
defined on $(\tOmega,\tFil,\tProb)$ such that
\begin{equation}\label{tqLG}
\tilde{\mu}(du,dt)=\tilde{\pi}_t(du)\,dt, \ \ \ \tProb-\mbox{a.s.}
\end{equation}
The filtration $\tilde{\mathds{F}}=\bigl\{\tFil_t\bigr\}_{t\in [0,T]}$ is defined by
\[
\tFil_t=\sigma\{(\tX_s,\tilde{\pi}_s):s\in[0,t]\}, \ \ \ \ t\in [0,T].
\]
We now claim that $\tilde Y$ is a $\tilde{\mathds{F}}$-martingale.
Indeed, 
From (\ref{tXn}) we have
\begin{equation}\label{tXn4}
\sup_{\, t\in[0,T]}\bigl|\tX^n_t-\tX_t\bigr|^2\to 0, \ \ \mbox{ as } n\to \infty, \ \ \tProb-\mbox{a.s.}
\end{equation}
 By Theorem \ref{a-priori}, Corollary \ref{S-H} and (\ref{eqlawsn}),  it follows that
\begin{equation}\label{tXest}
\tExp\biggl[\sup_{\, t\in [0,T]}\bigl|\tX_t\bigr|^{m}\biggr]<\infty.
\end{equation}
Also by Theorem \ref{a-priori},  and Chebyshev's inequality, the random variables in (\ref{tXn4}) are uniformly integrable. Then, by  Lemma 4.11 in  \cite{kallenberg} we have
\begin{equation}\label{tXnconv}
\tExp\biggl[\sup_{\, t\in[0,T]}\bigl|\tX^n_t-\tX_t\bigr|^2\biggr]\to 0, \ \ \mbox{ as } n\to \infty.
\end{equation}
Similarly, we have
\begin{equation}\label{tMnconv2}
\tExp\biggl[\sup_{\, t\in[0,T]}\bigl|\tilde Y^n_t-\tilde Y_t\bigr|^2\biggr]\to 0, \ \ \mbox{ as } n\to \infty.
\end{equation}
This, in conjunction with the martingale property of $\tilde Y^n,$ implies that for all $0<s<t\le T$ and for all
\[
\phi\in \mathcal{C}_b\left(\mathcal{C}(0,s;\mathbb{R}^{d})\times\Y(0,s;U)\right)
\]
we have that as $n\to\infty$
\[
0=\tExp\left[\bigl(\tilde{Y}^n_t-\tilde{Y}^n_s\bigr)\phi(\tX^n,\tilde{\mu}^n)\right]\to \tExp\left[\bigl(\tilde{Y}_t-\tilde{Y}_s\bigr)\phi(\tX,\tilde{\mu})\right],
\]  which implies that $\tilde Y$ is a $\tilde{\mathds{F}}-$martingale. 


\noindent{\sc Step 4.} We now pass to the limit to identify the  process $(\tX_t)_{t\in [0,T]}$ as a solution of the equation controlled by $(\tilde{\pi}_t)_{t\in [0,T]}$. Using Lemma \ref{weak-conv-coeff} with $E=\R^d$, $f=b,$ $\beta=p>1,$ $\delta=1$, and $\vartheta=\vartheta_1$  we obtain
\begin{equation}\label{tbntb}
\tb^n\rightharpoonup \tb, \ \ \ \mbox{ weakly in } \ L^1([0,T]\times\tOmega;\mathbb{R}^{d})
\end{equation}
with $\tilde b_t=\bar b(t,\tilde X_t,\tilde \pi_t), \ t\in[0,T].$ We claim that the process $\tilde{Y}$ satisfies
\begin{align}\label{tM2LG}
\int _0^t \tilde X_s ds=\int_0^t x_0(s) ds + \int_0^t K(t-s)\left( \int_0^s \tilde b_\tau\,d\tau + \tilde Y_s\right)\,ds.
\end{align}
By (\ref{tXnconv}) and (\ref{tMnconv2}), for any $\varepsilon>0$ there exists an integer $\bar{m}=\bar{m}(\varepsilon)\geq 1$ for which
\begin{equation}\label{convXM1}
\tExp\biggl[\sup_{\, t\in [0,T]}\bigl|\tX^n_t-\tX_t\bigr|+\bigl|\tilde{Y}^n_t-\tilde{Y}_t\bigr|\biggr]<\varepsilon, \ \ \ \forall n\geq\bar m.
\end{equation}
From (\ref{tbntb}) we have
\[
\tb\in\overline{\{\tb^{\bar{m}},\tb^{\bar{m}+1},\ldots\}}^{w}
\subset \overline{{\rm co}\{\tb^{\bar{m}},\tb^{\bar{m}+1},\ldots\}}^{w}
\]
where ${\rm co}(\cdot)$ and $\overline{\,\cdot\,}^{w}$ denote the convex hull and weak-closure in $L^1([0,T]\times\tOmega;\mathbb{R}^d)$ respectively. By Mazur's, see for example Theorem 2.5.16 in   \cite{megginson2008introduction}
\[
\overline{{\rm co}\{\tb^{\bar{m}},\tb^{\bar{m}+1},\ldots\}}^{w}
=\overline{{\rm co}\{\tb^{\bar{m}},\tb^{\bar{m}+1},\ldots\}}.
\]
Therefore, there exist an integer $\bar{N}\geq 1$ and $\{\alpha_1,\ldots,\alpha_{\bar{N}}\}$ with $\alpha_i\geq 0$,  $\sum_{i=1}^{\bar{N}}\alpha_i=1$,  such that
\begin{equation}\label{alphaif1}
\Bigl|\!\Bigl|\sum_{i=1}^{\bar{N}}\alpha_i\tb^{\bar m+i}-\tb\Bigr|\!\Bigr|_{L^1([0,T]\times\tOmega;\mathbb{R}^d)}<\varepsilon.
\end{equation}
Let $t\in [0,T]$ be fixed. Using the $\alpha_i$'s and (\ref{mart2}) we can write
\begin{align*}
\int_0^t x_0(s)ds=\sum_{i=1}^{\bar{N}}\alpha_i\left\{ \int_0^t \tX_s^{\bar{m}+i} ds - \int_0^tK(t-s)\left( \int_0^s \tb_v^{\bar{m}+i}\ dv+\tilde Y_s^{\bar{m}+i} \right) ds \right\}.
\end{align*}
Thus, we have
\begin{align*}
I&=\Bigl| \int_0^t x_0(s) ds+\int_0^tK(t-s)\left( \int_0^s \tb_v\, dv+\tilde Y_s \right) ds- \int_0^t\tX_s ds \Bigr|\\
&=\Bigl| \sum_{i=1}^{\bar{N}}\alpha_i\left\{ \int_0^t \tX_s^{\bar{m}+i} ds - \int_0^tK(t-s)\left( \int_0^s \tb_v^{\bar{m}+i}\ dv+\tilde Y_s^{\bar{m}+i} \right) ds \right\} \\
 &+\int_0^tK(t-s)\left( \int_0^s \tb_v\ dv+\tilde Y_s\right) ds- \int_0^t\tX_s ds  \Bigr|\\
&\le \Bigl| \sum_{i=1}^{\bar{N}}\alpha^i \int_0^t  \tX^{\bar m+i}_s ds - \int_0^t  \tX_s ds   \Bigr|+\Bigl|\sum_{i=1}^{\bar{N}}\alpha^i \int_0^t K(t-s)\tilde Y_s^{\bar{m}+i} ds -\int_0^t K(t-s)\tilde Y ds \Bigr| \\
 &+\,\Bigl|   \sum_{i=1}^{\bar{N}}\alpha^i\int_0^t K(t-s) \int_0^s\tb^{\bar m+i}_v\, dv\,ds -\int_0^tK(t-s) \int_0^s\tb_v\, dv\,ds\Bigr|=II+III+IV.
\end{align*}
Then, by  (\ref{convXM1}), we have
\begin{align*}
\tExp (II) &= \tExp\left[ \Bigl| \sum_{i=1}^{\bar{N}}\alpha^i \int_0^t  \tX^{\bar m+i}_s ds - \int_0^t  \tX_s ds \Bigr|\right]\le  \sum_{i=1}^{\bar{N}}\alpha^i\tExp\left[\bigl| \int_0^t  (\tX^{\bar m+i}_s  - \tX_s ) ds     \bigr|\right] \\
& \le \sum_{i=1}^{\bar{N}}\alpha^i\tExp\left[\int_0^t \sup_{s\in[0,T]}\bigl|\tX^{\bar m+i}_s-\tX_s \bigr| \, ds  \right] \le \varepsilon T.
\end{align*}
By  Fubini's theorem and (\ref{convXM1}), it follows
\begin{align*}
\tExp (III) &= \tExp\left[ \Bigl| \sum_{i=1}^{\bar{N}}\alpha^i  \int_0^t K(t-s)\tilde Y_s^{\bar{m}+i}\, ds -\int_0^t K(t-s)\tilde Y_s\,ds \Bigr|\right]\\
&\le  \sum_{i=1}^{\bar{N}}\alpha^i\tExp\left[\bigl| \int_0^t K(t-s) \left(\tilde Y_s^{\bar{m}+i}-\tilde Y_s\right)\,ds     \bigr|\right] \\
&\le  \sum_{i=1}^{\bar{N}}\alpha^i\tExp\left[\bigl| \int_0^t K(t-s) \left(\tilde Y_s^{\bar{m}+i}-\tilde Y_s\right)\, ds    \bigr|\right] \le \varepsilon \norm{K}_{L^1(0,T)}.
\end{align*}
Using twice Jensen's inequality and (\ref{alphaif1}),
\begin{align*}
\tExp (IV) &= \tExp\left[ \Bigl| \sum_{i=1}^{\bar{N}}\alpha^i \int_0^t K(t-s) \int_0^s\tb^{\bar m+i}_v\ dv\,ds -\int_0^tK(t-s) \int_0^s\tb_v\ dv\,ds \Bigr|\right]\\
&= \tExp\left[ \Bigl|  \int_0^t K(t-s) \int_0^s \Bigl(\sum_{i=1}^{\bar{N}}\alpha^i\tb^{\bar m+i}_v -\tb_v\Bigr)\,dv \,ds \Bigr|\right]\\
&\le  \tExp\left[t \int_0^t \abs{K(t-s)}\int_0^s  \Bigl|\sum_{i=1}^{\bar{N}}\alpha^i\tb^{\bar m+i}_v -\tb_v \Bigr|\,dv\,ds    \right]  \le \varepsilon T \norm{K}_{L^1(0,T)}.
\end{align*}
Then, $\tExp(I)\le [T+\norm{K}_{L^1(0,T)}(T+1)] \varepsilon.$ Since $\varepsilon>0$ is arbitrary, (\ref{tM2LG}) follows.

\noindent\textsc{Step 5.} Set $\tilde\sigma_t=\bar\sigma(t,\tX_t,\tilde{\pi}_t)$ and $\tilde\sigma^n_t=\bar\sigma(t,\tX^n_t,\tilde{\pi}^n_t)$ ,  $t\in [0,T]$. Using Lemma \ref{weak-conv-coeff} with $E=\mathbb{S}^{d}$, $f=\sigma\sigma^\top, \ \delta=2, \ \vartheta=\vartheta_1^{2}$ and $\beta= p/2>1,$  we obtain
\begin{equation*}
	\tilde\sigma^n\tilde\sigma^{n,\top}\rightharpoonup \tilde\sigma\tilde\sigma^\top, \ \ \ \mbox{ weakly in } \ L^1([0,T]\times\tOmega;\mathbb{S}^{d}).
\end{equation*}
Let $t\in[0,T]$ be fixed. By \eqref{tMnconv2} and the Burkholder-Davis-Gundy inequality we have $\bigl\langle\tilde Y^n\bigr\rangle_t\to\bigl\langle\tilde Y\bigr\rangle_t$ in $L^2(\tilde\Omega).$ Then, for any $\varepsilon>0$ there exists an integer $\bar{m}=\bar{m}(\varepsilon)\geq 1$ such that
\begin{equation*}
	\tExp\biggl[\bigl|\bigl\langle\tilde Y^n\bigr\rangle_t-\bigl\langle\tilde Y\bigr\rangle_t|\biggr]<\varepsilon, \ \ \ \forall n\geq\bar m.
\end{equation*}
As in the proof of Step 5, there also exists an integer $\bar{N}\geq 1$ and $\{\alpha_1,\ldots,\alpha_{\bar{N}}\}$ with $\alpha_i\geq 0$,  $\sum_{i=1}^{\bar{N}}\alpha_i=1$, such that
\begin{equation*}
	\Bigl|\!\Bigl|\sum_{i=1}^{\bar{N}}\alpha_i\tilde\sigma^{\bar m+i}\tilde\sigma^{(\bar m+i),\top}-\tilde\sigma\tilde\sigma^\top\Bigr|\!\Bigr|_{L^1([0,T]\times\tOmega;\mathbb{S}^m)}<\varepsilon.
\end{equation*}
Thus, we have
\begin{align*}
	&\Bigl|\bigl\langle\tilde Y\bigr\rangle_t-\int_0^t\tilde\sigma_s\tilde\sigma_s^\top\,ds\Bigr|\\
	&=\Bigl|\bigl\langle\tilde Y\bigr\rangle_t-\sum_{i=1}^{\bar{N}}\alpha^i\bigl\langle\tilde Y^{\bar m+i}\bigr\rangle_t
	+\sum_{i=1}^{\bar{N}}\alpha^i\int_0^t\tilde\sigma^{\bar m+i}_s\tilde\sigma_s^{(\bar m+i),\top}\,ds-\int_0^t\tilde\sigma_s\tilde\sigma_s^\top\,ds\Bigr|\\
	&\le \Bigl|\bigl\langle\tilde Y\bigr\rangle_t-\sum_{i=1}^{\bar{N}}\alpha^i\bigl\langle\tilde Y^{\bar m+i}\bigr\rangle_t\Bigr|
	+\Bigl|\sum_{i=1}^{\bar{N}}\alpha^i \int_0^t\tilde\sigma^{\bar m+i}_s\tilde\sigma_s^{(\bar m+i),\top}\,ds-\int_0^t\tilde\sigma_s\tilde\sigma_s^\top\,ds
	\Bigr|.
\end{align*}
As in Step 4, we have
\[
\tExp\biggl[\Bigl|\bigl\langle\tilde Y\bigr\rangle_t-\int_0^t\tilde\sigma_s\tilde\sigma_s^\top\,ds\Bigr|\biggr]<(1+T)\varepsilon.
\]
Since $\varepsilon>0$ and $t\in[0,T]$ are arbitrary, it follows $\bigl\langle\tilde Y\bigr\rangle_t=\int_0^t\tilde\sigma_s\tilde\sigma_s^\top\,ds \ \tProb$-a.s. for all $t\in[0,T].$ By the martingale representation theorem (see e.g. Theorem  4.2 in Chapter 3.4 \cite{karatzas2012brownian}) there exist an extension of the probability space $(\tOmega,\tFil,\tProb)$,  which we also denote $(\tOmega,\tFil,\tProb)$,  and a $d'$-dimensional Brownian motion $(\tilde{W}_t)_{t\ge 0}$ defined on $(\tOmega,\tFil,\tProb)$,  such that
\[
\tilde{Y}_t=\int_0^t \bar\sigma(s,\tX_s,\tilde{\pi}_s)\,d\tilde {W}_s, \ \ \ \tProb-\mbox{a.s.}, \ \ \ t\in [0,T],
\]
By (\ref{tM2LG}), it follows that
\[
\int _0^t \tilde X_s ds=\int_0^t x_0(s) ds + \int_0^t K(t-s)\left( \int_0^s \tilde b_v \,dv + \tilde Y_s\right)ds,\
\tProb-\mbox{a.s.}
\]
for each $t\in [0,T]$. By Lemma \ref{equiv-Ev} this is equivalent to $\tX$ being a solution to the stochastic Volterra equation controlled by $\tilde{\pi}.$ In other words, $\tilde\Theta=\left(\tilde\Omega,\tilde{\mathcal F}, \tilde\Prob,\tilde{\mathds F}, \tilde W,\tilde X,\tilde \pi\right)$
is a weak admissible relaxed control. By the Fiber Product Lemma \ref{Lem:fibpro} we have
\[
\underline{\delta}_{\tX^n}\otimes\mu_n\to \underline{\delta}_{\tX}\otimes\mu, \ \mbox{ stably in } \ \Y(0,T;\R\times U), \ \ \tProb-\mbox{a.s.}
\]
Since $\mathbb{R}\times U$ is also a metrisable Suslin space, using Lemma \ref{Lem:lscym} and Fatou's Lemma we get
\[
\tExp\int_{U\times[0,T]} l(t,\tX_t,u)\,\tilde\mu(du,dt)
\le\liminf_{n\to\infty}\tExp\int_{U\times[0,T]} l(t,\tX^n_t,u)\,\tilde \mu_n(du,dt)\]
and since $(\tX^n,\tilde{\mu}^n)\stackrel{d}{=}(X^n,\mu^n)$ it follows that
\begin{eqnarray*}
\bar{\mathcal{J}}(\tilde\pi)&=&\tExp\int_{U\times[0,T]} l(t,\tX_t,u)\,\tilde{\mu}(du,dt)+\tExp G(\tX_T)\\
&\le&\liminf_{n\to\infty}\Exp^n\int_{U\times[0,T]}
l(t,X^n_t,u)\,\mu_n(du,dt)+\liminf_{n\to\infty}\Exp^n G(X^n_T)\\
&\le& \liminf_{n\to\infty}\left[\Exp^n\int_{U\times[0,T]} l(t,X^n_t,u)\,\mu_n(du,dt)+\Exp^nG(X^n_T)\right]=\inf_{\Theta\in\mathcal{U}(x_0)}\bar{\mathcal{J}}(\Theta),
\end{eqnarray*}
that is, $\tilde\Theta$ is a weak optimal relaxed control for \textbf{(RCP)}, and this concludes the proof of Theorem \ref{thm-main}.

\subsection{Proof of Theorem \ref{thm-strict}}
\begin{proof} Let $\Theta=\left(\Omega,\mathcal F, \mathbf P,\mathds F,X,W,\pi\right)\in\bar{\U}(x_0,T).$ Define $\mathcal{E}:[0,T]\times\Omega\to\R^{d\times d}\times\R^d\times\R$ as
\[
\mathcal{E} (t, \omega):=(\bar\sigma\bar\sigma^\top,\bar{b},\bar{h})(t,X_t(\omega),\pi_t(\omega)).
\]
By Assumption \ref{A3}, we have that   $\mathcal{E}(t, \omega) \in \Gamma(t,X_t(\omega)),$ as defined in \eqref{eq:Gamma}, for all $(t,\omega)\in [0,T]\times\Omega$. We also define  
\[
c^1(t,\omega):=(\bar\sigma\bar\sigma^\top,\bar{b})(t,X_t(\omega),\pi_t(\omega)),\quad c^2(t,\omega):=\bar{l}(t,X_t(\omega),\pi_t(\omega)).
\] 
By Lemma \ref{Lem:disrym}, $c^1$ and $c^2$ are measurable with respect to the predictable $\sigma$-algebra $\mathcal G$ on $Y=[0,T]\times\Omega.$  
Using Theorem \ref{Filipov}  we conclude the existence of a function $u:[0,T]\times\Omega\to U$ measurable with respect to $\mathcal G$ such that
\begin{equation}\label{eq:measurable_selection}
c^1(t,\omega)=(\sigma\sigma^\top,b)(t,X_t(\omega),u_t(\omega)),\quad c^2(t,\omega)\geq l(t,X_t(\omega),u_t(\omega)), \quad (t,\omega)\in Y
\end{equation}
and the desired result follows. \end{proof}

\appendix
\section{Auxiliary results}\label{A:A}

Let $\left( Y,\mathcal G, \mu\right)$ be a measure space, $k,m$ be natural numbers, and  $U$ a closed subset of an Euclidean space. Let 
\[
c^1:Y\rightarrow \R^{k}, \quad c^2:Y\rightarrow \R^{m}, \ \ \phi:Y\times U\rightarrow \R^{k}, \ \ \psi:Y\times U\rightarrow \R^{m}_{+},
\]
be given measurable functions with $u\rightarrow \phi(y,u)$ continuous and $u\rightarrow \psi_i(y,u)$ lower semi-continuous, for each $y\in Y$ and $i=1,2,...m.$ Define
\[
\Gamma(y,U)=\left\{ ( \phi(y,u),z ) \in \R^{k} \times \R^{m}: u \in U, z_i \geq \psi_i(y,u)\text{ for $i=1,\ldots,m$} \right\}.
\]

\begin{theorem}
\label{Filipov}
If $(c^1(y),c^2(y))\in \Gamma(y,U)$ for all $y\in Y$, then there exists a measurable function $u:Y \rightarrow U $ such that
$c^1(y)=\phi(y,u(y))$ and $c^2_i(y)\geq \psi_i(y,u(y)),$ $i=1,\dots,m$.
\end{theorem}
\begin{proof} \cite[Theorem A.9]{haussmann}. \end{proof}
\section{Relative compactness and limit theorems for Young measures}\label{A:B}
Young measures on metrizable Suslin control sets have been studied by  \cite{balder2} and  \cite{deFitte1}. We refer to the book  \cite{cadfval} for more details. 

\begin{proposition}\label{stsuslinmet}
	Let $U$ be metrisable (resp. metrisable Suslin). Then the space $\Y(0,T;U)$ endowed with the stable topology is also metrizable (resp. metrizable Suslin).
\end{proposition}
 \begin{proof}
 Propositions 2.3.1 and 2.3.3 in \cite{cadfval}. \end{proof} 

\textcolor{black}{
The notion of tightness for Young measures that we use was introduced by  \cite{val90}. See also the book  \cite{crauel}. Recall that a set-valued function $[0,T]\ni t\mapsto K_t\subset U$ is said to be \emph{measurable} if and only if for every open set $\tilde U\subset U$,
\[
\set{t\in [0,T]: K_t\cap \tilde U\neq\varnothing}\in\B([0,T]).
\]
}
\textcolor{black}{
\begin{definition}
	We say that a set $\frakJ\subset\Y(0,T;U)$  is \emph{flexibly tight} if,
	for each $\varepsilon>0$,  there exists a measurable set-valued mapping $[0,T]\ni t\mapsto K_t\subset U$ such that $K_t$ is compact for all $t\in [0,T]$ and
	\[
	\sup_{\mu\in\frakJ}\,\int_{U\times[0,T]} \mathbf{1}_{K_t^c}(u)\,\mu(du,dt)<\varepsilon.
	\]
\end{definition} 
}

\begin{theorem}[\textbf{Equivalence theorem for flexible tightness}]\label{flexiblytight}
	For any $\frakJ\subset\Y(0,T;U)$ the two following conditions are equivalent:
	\begin{enumerate}
		\item $\frakJ$ is flexibly tight
		\item There exists $\vartheta\in IC(0,T;U)$ such that
		\[
		\sup_{\mu\in\frakJ}\,\int_{U\times[0,T]}\vartheta(t,u)\,\mu(du,dt)< +\infty.
		\]
	\end{enumerate}
\end{theorem}
\begin{proof}
See the equivalence assertion in  \cite[Definition 3.3]{balder1998lectures}. \end{proof}

\begin{theorem}[\textbf{Prohorov criterion for relative compactness},]\label{prohym}
	Let $U$ be a metrisable Suslin space. Then every flexibly tight subset of $\Y(0,T;U)$ is sequentially relatively compact in the stable topology.
\end{theorem}
\begin{proof}
\cite[Theorem 4.3.5]{cadfval}
\end{proof}

\begin{lemma}\label{Lem:lscym}
	Let $U$ be a metrisable Suslin space and $G\in L^1(0,T;\R)$. Let us assume that
	\[
	l:[0,T]\times U\to [-\infty,+\infty]
	\]
	is a measurable function such that $l(t,\cdot)$ is lower semi-continuous for every $t\in [0,T]$ and satisfies one of the two following conditions:
	\begin{enumerate}
		\item $\abs{l(t,u)}\le G(t)$,  a.e. $t\in [0,T]$,
		\item $l\geq 0$.
	\end{enumerate}
	If $\mu_n\to\mu$ stably in $\Y(0,T;U)$,  then
	\[
	\int_{U\times[0,T]} l(t,u)\,\mu(du,dt)\le \liminf_{n\to\infty}\int_{U\times[0,T]} l(t,u)\,\mu_n(du,dt).
	\]
\end{lemma}
\begin{proof}
\cite[Lemma 2.15]{brzezniak2013optimal}
\end{proof}

It is worth mentioning that these last two results are, in fact, the main reasons why it suffices for the control set $U$ to be only metrizable and Suslin, in contrast with the existing literature on stochastic relaxed controls. Indeed, Theorem \ref{prohym} is key to obtain tightness of the laws of random Young measures in the proof of the main existence result, and Lemma \ref{Lem:lscym} is used to prove the lower semi-continuity of the relaxed cost functionals as well as Theorem \ref{thm:contym} below.

\begin{theorem}\label{thm:contym}
	Let $U$ be a metrisable Suslin space. If $\mu_n\to\mu$ stably in $\Y(0,T;U)$,  then for every $f\in L^1(0,T;\mathcal{C}_b(U))$ we have
	\[
	\lim_{n\to\infty}\int_{U\times[0,T]}f(t,u)\,\mu_n(du,dt)=\int_{U\times[0,T]} f(t,u)\,\mu(dt,du).
	\]
\end{theorem}
 \begin{proof} Use Lemma \ref{Lem:lscym} with $f$ and $-f$. \end{proof}
	
We will need the following version of the so-called Fiber Product Lemma. The proof can be found in \cite[Lemma 2.17]{brzezniak2013optimal}.  For a measurable map $y:[0,T]\to U$,  we denote by $\underline{\delta}_{y(\cdot)}(\cdot)$ the \emph{degenerate Young measure} defined as $\underline{\delta}_{y(\cdot)}(du,dt)=\delta_{y(t)}(du)\,dt$.
\begin{lemma}[{\bf Fiber Product Lemma}]\label{Lem:fibpro}
	Let $\mathcal{S}$ and $U$ be separable metric spaces and let $y_n:[0,T]\to\mathcal{S}$ be a sequence of measurable mappings which converge pointwise to a mapping $y:[0,T]\to \mathcal{S}$. Let $\mu_n\to\mu$ stably in $\Y(0,T;U)$ and consider the following sequence of Young measures on $\mathcal{S}\times U$:
	\[
	(\underline{\delta}_{y_n}\otimes\mu_n)(dx,du,dt)=\delta_{y_n(t)}(dx)\,\mu_n(du,dt), \ \ n\in\mathds{N},
	\]
	and
	\[
	(\underline{\delta}_{y}\otimes\mu)(dx,du,dt)=\delta_{y(t)}(dx)\,\mu(du,dt).
	\]
	Then $\underline{\delta}_{y_n}\otimes\mu_n\to \underline{\delta}_{y}\otimes\mu$ stably in $\Y(0,T;S\times U)$.
\end{lemma}

\section*{Acknowledgement}
The first and third authors thank the Alianza EFI-Colombia Cientifica grant, codes 60185 and FP44842-220-2018 for financial support.
The research of the second author benefited from the financial support of the chairs ``Deep Finance \& Statistics'' and ``Machine Learning \& Systematic Methods in Finance'' of \'Ecole Polytechnique. The second author acknowledges support from the Europlace Institute of Finance (EIF) and the Labex Louis Bachelier, research project: ``The impact of information on financial markets'', and from the MATH AmSud project VOS 22-MATH-08 and the ECOS program C21E07. 

\bibliographystyle{plain}
\bibliography{biblioCSVE}

\end{document}